\documentclass[a4paper,11pt]{amsart}
\usepackage[latin1]{inputenc}
\usepackage{amsmath}
\usepackage{amsfonts}
\usepackage{amssymb}
\usepackage{amsthm}
\usepackage{mathrsfs}
\usepackage{graphicx}
\usepackage{comment}

\newcommand{\R}{{\mathbb{R}}}
\newcommand{\Q}{{\mathbb{Q}}}
\newcommand{\Z}{{\mathbb{Z}}}
\newcommand{\C}{\mathbb{C}}

\renewcommand\Re{\operatorname{Re}}
\renewcommand\Im{\operatorname{Im}}

\def\vecv{{\text{\boldmath$v$}}}
\def\vecm{{\text{\boldmath$m$}}}

\def\vec0{{\text{\boldmath$0$}}}

\newcommand{\ve}{\varepsilon}

\newcommand{\sfrac}[2]{{\textstyle \frac {#1}{#2}}}

\newcommand{\SL}{\mathrm{SL}}

\newcommand{\GL}{\mathrm{GL}}

\newcommand{\gammamax}{\Gamma_{\text{max}}}

\newtheorem{thm}{Theorem}[section]
\newtheorem{lem}[thm]{Lemma}
\newtheorem{prop}[thm]{Proposition}
\newtheorem{cor}[thm]{Corollary}

\theoremstyle{remark}
\newtheorem{remark}[thm]{Remark}

\numberwithin{equation}{section}

\begin{document}
%opening
\title{On the universality of the Epstein zeta function}
\author{Johan Andersson and Anders S\"odergren}
\address{Department of Mathematics, School of Science and Technology, {\"O}rebro\newline
\rule[0ex]{0ex}{0ex}\hspace{12pt}University, {\"O}rebro, SE-701 82 Sweden \newline
\rule[0ex]{0ex}{0ex}\hspace{8pt} {\tt johan.andersson@oru.se}\newline
\newline
\rule[0ex]{0ex}{0ex}\hspace{12pt}Department of Mathematical Sciences, University of Copenhagen,\newline
\rule[0ex]{0ex}{0ex}\hspace{12pt}Universitetsparken 5, 2100 Copenhagen \O, Denmark \newline
\rule[0ex]{0ex}{0ex}\hspace{12pt}\textit{Present address}: Department of Mathematical Sciences, Chalmers University of \newline 
\rule[0ex]{0ex}{0ex}\hspace{12pt}Technology and the University of Gothenburg, SE-412 96 Gothenburg, Sweden \newline
\rule[0ex]{0ex}{0ex}\hspace{12pt}{\tt andesod@chalmers.se}} 


\date{\today}
\thanks{The second author was partially supported by a postdoctoral fellowship from the Swedish Research Council, by the National Science Foundation under agreement No.\ DMS-1128155, as well as by a grant from the Danish Council for Independent Research and FP7 Marie Curie Actions-COFUND (grant id: DFF-1325-00058).}

\maketitle

\begin{abstract}
We study universality properties of the Epstein zeta function $E_n(L,s)$ for lattices $L$ of large dimension $n$ and suitable regions of complex numbers $s$. Our main result is that, as $n\to\infty$, $E_n(L,s)$ is universal in the right half of the critical strip as $L$ varies over all $n$-dimensional lattices $L$. The proof uses a novel combination of an approximation result for Dirichlet polynomials, a recent result on the distribution of lengths of lattice vectors in a random lattice of large dimension and a strong uniform estimate for the error term in the generalized circle problem. Using the same approach we also prove that, as $n\to\infty$, $E_n(L_1,s)-E_n(L_2,s)$ is universal in the full half-plane to the right of the critical line as $(L_1,L_2)$ varies over all pairs of $n$-dimensional lattices. Finally, we prove a more classical universality result for $E_n(L,s)$ in the $s$-variable valid for almost all lattices $L$ of dimension $n$. As part of the proof we obtain a strong bound of $E_n(L,s)$ on the critical line that is subconvex for $n\geq 5$ and almost all $n$-dimensional lattices~$L$.
\end{abstract}

\section{Introduction}

In 1975 Voronin \cite{voronin1, voronin2} proved the following remarkable approximation theorem
for the Riemann zeta function:

\begin{thm}[Voronin]\label{Voroninstheorem}  
Let $K=\left\{ s \in \C : \left|s-\frac34\right| \leq r\right\}$ for some $r<\frac14$, and suppose that $f$
is any nonvanishing continuous function on $K$ that is analytic  in the
interior of $K$. Then, for any $\ve>0$, 
$$\liminf_{T \to \infty} \frac 1 T \mathop{\rm meas} \Big \{t \in
[0,T]:\max_{s \in K} \big|\zeta(s+it)-f(s)\big|<\varepsilon \Big \}>0. $$
\end{thm}

This theorem, known as \emph{Voronin's Universality Theorem}, shows that any nonvanishing analytic function in a small disc may be
approximated by a vertical shift of the Riemann zeta function. It has been improved and generalized in various directions (see Steuding's monograph \cite{steuding} and Matsumoto's survey paper \cite{matsumoto} for detailed discussions); for example it is known that the set $K$ may be chosen as any compact set with connected complement that lies in the vertical strip \begin{gather}\label{Ddef} 
D:=\left\{s \in \C: \tfrac12<\Re(s)<1\right\}. 
\end{gather} 
Let us also note that, since $\zeta(s)$ has an Euler product, it is necessary to assume that
the function $f$ in Theorem \ref{Voroninstheorem} is nonvanishing on $K$. However, for zeta functions without
an Euler product, such as the Hurwitz zeta function $\zeta(s,\alpha)$ with
rational or transcendental\footnote{For the case of algebraic irrational $\alpha$, the Hurwitz zeta function does not have an Euler product and it is likely to be universal without assuming the nonvanishing condition. However, proving this seems quite difficult and constitutes a major open problem in the theory of universality.} parameter $0<\alpha<1$, $\alpha \neq \frac12$, this condition can be removed. 

Similar universality theorems have been proved for large classes of zeta functions and $L$-functions. Here we will content ourselves with a short review of the situation for Dirichlet $L$-functions. In his thesis, Voronin \cite{voroninthesis} proved\footnote{Similar results were established (independently) by Gonek \cite{gonek} and Bagchi \cite{BagchiPhD}.} the \emph{joint universality} of Dirichlet $L$-functions, i.e.\ that vertical shifts of Dirichlet $L$-functions attached to nonequivalent Dirichlet characters can be used to simultaneously approximate any finite number of nonvanishing analytic functions on $K$. In a different direction, Bagchi \cite{BagchiPhD} has proved a universality theorem for Dirichlet $L$-functions $L(s,\chi)$ in which the imaginary shifts in the complex argument $s$ from Theorem \ref{Voroninstheorem} has been replaced by a variation of the character $\chi$ over the set of characters of prime modulus. To be precise, let $K$ be a compact subset of $D$ with connected complement and let $f$ be a nonvanishing continuous function on $K$ that is analytic in the
interior of $K$. Then, for any $\ve>0$, 
\begin{align}\label{Bagchisresult}
\liminf_{p \to \infty} \frac{1}{\varphi(p)}\#\Big \{\chi \hspace{-7pt}\mod p : \max_{s \in K} \left|L(s,\chi)-f(s)\right|<\varepsilon \Big \}>0. 
\end{align}

Recently Mishou and Nagoshi \cite{MiNa1, MiNa2} proved a related result
for Dirichlet $L$-functions associated with real characters. Let $\mathcal D^+$ (resp.\ $\mathcal D^-$) denote the set of positive fundamental discriminants (resp.\ negative fundamental discriminants) and define $$\mathcal D^{\pm}(X):=\big\{d\in\mathcal D^{\pm} : |d|\leq X\big\}.$$ For a discriminant $d$, we let $\chi_d$ denote the quadratic Dirichlet character
modulo $|d|$ defined by the Kronecker symbol $\chi_d(n)=(\frac dn)$. 

\begin{thm}[Mishou-Nagoshi]\label{MN-theorem}
Let $\Omega$ be a simply connected
region in $D$ that is symmetric with respect to the real axis. Then, for any
$\varepsilon>0$, any compact set $K\subset\Omega$ and any nonvanishing holomorphic
function $f$ on $\Omega$ which takes positive real values on
$\Omega\cap\R$, we have
\begin{align*}
\liminf_{X \to \infty}\frac{1}{\#\mathcal D^{\pm}(X)} \# \Big\{d\in\mathcal D^{\pm}(X) : \max_{s \in
K} \left|L(s,\chi_d)-f(s)\right|<\varepsilon  \Big\}>0.  
\end{align*}
\end{thm}

The purpose of the present paper is to prove universality results for the Epstein zeta function that turn out to have similarities with both Theorem \ref{MN-theorem} and Bagchi's result described in \eqref{Bagchisresult} above. In order to state and describe our results, we first need to properly introduce the setting.

Let $X_n$ denote the space of all $n$-dimensional lattices $L\subset\R^n$ of covolume one and let $\mu_n$ denote Siegel's measure \cite{siegel} on $X_n$, normalized to be a probability measure. For $L\in X_n$ and $\Re(s)>\frac{n}{2}$, the Epstein zeta function is defined by
\begin{align*}
E_n(L,s):=\sum_{\vecm\in L\setminus\{\vec0\}}|\vecm|^{-2s}\,.
\end{align*}
$E_n(L,s)$ has an analytic continuation to $\C$ except for a simple pole at $s=\frac{n}{2}$ with residue $\pi^{\frac n2}\Gamma(\frac{n}{2})^{-1}$. Furthermore, $E_n(L,s)$ satisfies the functional equation
\begin{align}\label{FUNCTIONALEQ}
F_n(L,s)=F_n(L^*,\sfrac{n}{2}-s),
\end{align}
where $F_n(L,s):=\pi^{-s}\Gamma(s)E_n(L,s)$ and $L^*$ is the dual lattice of $L$. The Epstein zeta function has many properties in common with the Riemann zeta function. In fact, the functions $E_n(L,s)$ (actually a slightly more general family of functions) were introduced by Epstein \cite{paulepstein1,paulepstein2} in an attempt to find the most general form of a function satisfying a functional equation of the same type as the Riemann zeta function. Note in particular the relation
\begin{align*}
E_1(\Z,s)=2\zeta(2s).
\end{align*}
However, we stress that there are also important differences between $E_n(L,s)$ and $\zeta(s)$. Typically $E_n(L,s)$ has no Euler product and it is well known that the Riemann hypothesis for $E_n(L,s)$ generally fails (cf., e.g., \cite{SodStrom} and the references therein).\footnote{Here the words \emph{typically} and \emph{generally} are to be interpreted in terms of the measure $\mu_n$.} 

Let $V_n$ denote the volume of the $n$-dimensional unit ball. We recall the explicit formula
\begin{equation*}
V_n=\frac{\pi^{n/2}}{\Gamma(\frac n2+1)}\,,
\end{equation*} 
and stress that $V_n$ decays extremely fast as $n\to\infty$. In most of our results $V_n$ will appear naturally as part of a factor normalizing $E_n(L,s)$. 

Our first main result is a universality theorem for $E_n(L,s)$ in the lattice aspect, i.e.\ a universality result where the lattice $L$ varies over the space $X_n$ but no vertical shifts are applied to the complex variable $s$. The situation is related to the one in Bagchi's theorem \eqref{Bagchisresult}, and similarly, in order to obtain still finer approximations it is natural to consider the limit $n\to\infty$. Let us also point out that the relation with Theorem \ref{MN-theorem} lies in the fact that both $L(s,\chi_d)$ and $E_n(L,s)$ are real-valued for real values of $s$, resulting in the same sort of conditions on the involved functions and regions. However, since $E_n(L,s)$ typically has no Euler product, the nonvanishing condition on the function $f$ can be removed.  

\begin{thm}\label{UNIV1}
Let $\Omega$ be a simply connected region in $D$ that is symmetric with respect to the real axis. Then, for any $\ve>0$, any compact set $K\subset\Omega$ and any holomorphic function $f$ on $\Omega$ that is real-valued on $\Omega\cap\R$, we have
\begin{align*}
\liminf_{n\to\infty}{\rm Prob}_{\mu_n}\left\{L\in X_n : \max _{s\in K}\left|2^{s-1}V_n^{-s}E_n\left(L,\frac{ns}{2}\right)-f(s)\right|<\ve\right\}>0.
\end{align*}
\end{thm}

\begin{remark}
In particular it follows from Theorem \ref{UNIV1} that, given $\ve>0$, $K\subset\Omega$ and $f$ as above, there exists some $n\in\Z^+$ and a lattice $L\in X_n$ such that 
\begin{align*}
\max _{s\in K}\left|2^{s-1}V_n^{-s}E_n\left(L,\frac{ns}{2}\right)-f(s)\right|<\ve.
\end{align*}
\end{remark}

As an immediate consequence of Theorem \ref{UNIV1}, we have the following denseness result.

\begin{cor}
For any fixed $s\in D\setminus (D\cap\R)$ the set $$\left\{2^{s-1}V_n^{-s}E_n\left(L,\frac{ns}{2}\right):n\in\Z^+ , L\in X_n\right\}$$ is dense in $\C$. Moreover, for any fixed $x\in D\cap\R$, we have 
\begin{equation}\label{MORETHANDENSE}
\left\{2^{x-1}V_n^{-x}E_n\left(L,\frac{nx}{2}\right):n\in\Z^+ , L\in X_n\right\}=\R.
\end{equation}  
\end{cor}

To prove the first part of the corollary, it is sufficient to note that for any $c \in \C$ we may use  Theorem \ref{UNIV1} to approximate $f(s)=c$ on the set $K=\{s\}$ to arbitrary precision.
Next, to prove \eqref{MORETHANDENSE} we note that for any $N \in\Z^+$ we may use Theorem \ref{UNIV1}, with $\varepsilon=1$ and $K=\{x\}$, to approximate the two functions $f_1(s)=-N-1$ and $f_2(s)=N+1$. It follows that for all sufficiently large $n$ (depending on $N$) there exist lattices $L_1,L_2 \in X_n$ such that $2^{x-1}V_n^{-x}E_n\left(L_1,\frac{nx}{2}\right)<-N$ and $2^{x-1}V_n^{-x}E_n\left(L_2,\frac{nx}{2}\right)>N$. We note that in contrast to the situation in \cite[Corollary 1.2]{MiNa1}, the lattice variable $L$ varies continuously in $X_n$ and, moreover, the Epstein zeta function is continuous in the lattice variable.\footnote{In fact, $E_n(L,s)$ is an Eisenstein series associated to a maximal parabolic subgroup of $\SL(n,\Z)$ (see, e.g., \cite[Section 10.7]{goldfeld}). See also \cite[Theorem 2]{terras}.} Hence, since $X_n$ is connected and the Epstein zeta function is real-valued on the real line, the intermediate value theorem implies that $[-N,N] \subseteq \{2^{x-1}V_n^{-x}E_n\left(L,\frac{nx}{2}\right): L \in X_n \}$ for all sufficiently large $n$. This confirms that the left-hand side of \eqref{MORETHANDENSE} is not only dense in $\R$, but in fact equal to $\R$.

Let us also point out that Theorem \ref{UNIV1}, together with the Weierstrass approximation theorem, readily gives a universality result for approximation of continuous functions on closed subintervals of $(\frac12,1)$.  

\begin{cor}\label{REALUNIV}
Let $\frac12<a<b<1$. Then, for any $\ve>0$ and any continuous real-valued function $f$ on $[a,b]$, we have
\begin{align*}
\liminf_{n\to\infty}{\rm Prob}_{\mu_n}\left\{L\in X_n : \max _{x\in [a,b]}\left|2^{x-1}V_n^{-x}E_n\left(L,\frac{nx}{2}\right)-f(x)\right|<\ve\right\}>0.
\end{align*}
\end{cor}

We will now use Corollary \ref{REALUNIV} to give a short and more elementary proof of (one part of) \cite[Corollary 1.5]{epstein2}. Let $\ve<1$ and let, for any $\delta\in(0,\frac14)$,  $K_{\delta}$ denote the closed interval $[\frac12+\delta,1-\delta]\subset (\frac12,1)$. Now, by applying Corollary \ref{REALUNIV} with the continuous function $f(x)=-1$ on $K_{\delta}$, we obtain the following result.

\begin{cor}\label{NONZERO}
For any $\delta\in(0,\frac14)$, we have
\begin{align*}
\liminf_{n\to\infty}{\rm Prob}_{\mu_n}\left\{L\in X_n\,:\, E_n\left(L,\frac{nx}{2}\right)<0 \text{  for all  } x\in[\sfrac12+\delta,1-\delta]\right\}>0.
\end{align*} 
\end{cor}

\begin{remark}
It follows from \cite[Corollary 1.5]{epstein2} that we can replace the $\liminf$ in Corollary \ref{NONZERO} by a proper limit. Let us also point out that it follows from \cite[Corollary 1.7]{epstein2} that this limit probability tends to zero as $\delta$ tends to zero.
\end{remark}

Our next corollary discusses zeros of $E_n\left(L,\frac{nx}{2}\right)$ in the interval $\frac12<x<1$. Its proof is an application of Theorem \ref{UNIV1} with the function $f(s)=\sin\big(\pi\big(\frac{m(s-a)}{b-a}+\frac12\big)\big)$, together with Rouche's theorem and the observation that since $E_n\left(L,\frac{nx}{2}\right)$ is real-valued the zeros of $E_n\left(L,\frac{ns}{2}\right)$ approximating the zeros of $f(s)$ must remain on the real line.\footnote{If we allow the number of zeros to also be greater than $m$, then this result is an immediate consequence of Corollary \ref{REALUNIV}.}

\begin{cor}\label{ZEROS}
Let $\frac12<a<b<1$ and let $m\in\Z^+$. Then we have 
\begin{align*}
\liminf_{n\to\infty}{\rm Prob}_{\mu_n}\left\{L\in X_n : E_n\left(L,\frac{nx}{2}\right) \text{has exactly $m$ zeros in $[a,b]$}\right\}>0.
\end{align*}
\end{cor}

Let us give a brief sketch of the proof of Theorem \ref{UNIV1}, which is given in full detail in Section \ref{Proof1}. To begin, we decompose the expression under consideration into two parts as
\begin{align}\label{elementarysplitting}
\left(2^{s-1}V_n^{-s}E_n\left(L,\frac{ns}{2}\right)-\zeta(s)\right)+\big(\zeta(s)-f(s)\big),
\end{align}
and approximate the second part, i.e.\ $\zeta(s)-f(s)$ (which is holomorphic on the compact set $K$), to any desired accuracy by a Dirichlet polynomial. We continue by dividing the first part of \eqref{elementarysplitting} into a main term and a tail term. Using a recent result of the second author \cite{poisson} on the distribution of lengths of lattice vectors in a random lattice of large dimension, we show that the main term approximate the Dirichlet polynomial found above as well as desired for a positive proportion of lattices $L\in X_n$. To conclude the proof we use a strong uniform estimate of the error term in the generalized circle problem (cf.\ \cite{epstein2,schmidt}; see also Section \ref{Poissonsection}) to show that the tail term is sufficiently small.

Using the same general idea of proof, we arrive at our second main result. As far as we are aware, this is the first example of a universality theorem that is valid in the half-plane of absolute convergence (of the Dirichlet series under consideration).

\begin{thm}\label{UNIV2}
Let $\Omega$ be a simply connected region in the half-plane $\Re(s)>\frac12$ that is symmetric with respect to the real axis. Then, for any $\ve>0$, any compact set $K\subset\Omega$ and any holomorphic function $f$ on $\Omega$ that is real-valued on $\Omega\cap\R$, we have
\begin{align*}
\liminf_{n\to\infty}{\rm Prob}&_{\mu_n\times \mu_n}\bigg\{(L_1,L_2)\in X_n^2 : \\
&\max _{s\in K}\left|2^{s-1}V_n^{-s}E_n\left(L_1,\frac{ns}{2}\right)-2^{s-1}V_n^{-s}E_n\left(L_2,\frac{ns}{2}\right)-f(s)\right|<\ve\bigg\}>0.
\end{align*}
\end{thm} 

Classical Voronin universality (variants of Theorem 1.1) for the Epstein zeta function follows in some special cases, such as when the Epstein zeta function is (a constant multiple of) a Dedekind zeta function of an imaginary quadratic number field, since then it is a product of $\zeta(s)$ and a Dirichlet $L$-function, and universality follows from the joint universality of Dirichlet $L$-functions. Universality has also been proved in the case when the Epstein zeta function can be written as a linear combination of Hecke $L$-functions by joint universality for such functions, see \cite[pp.\ 279-284]{KarVor} and \cite{voronin3}. The general case seems to be more difficult. While we cannot prove classical universality for $E_n(L,s)$ with any single given lattice, we will prove universality for almost every lattice $L \in X_n$. At present time our methods produce a result in the strip $\frac34<\Re(s)<1$.\footnote{We note that a result recently stated by Blomer \cite[p.\ 2]{blomer} implies that Theorem \ref{UNIV3} holds for all compact sets $K\subset D=\big\{s \in \C: \frac12<\Re(s)<1 \big\}$ with connected complement.}

\begin{thm}\label{UNIV3}  
Let $n \geq 2$. Then, for almost all $L \in X_n$ and for any $\ve>0$, any compact set $K \subset \big\{s \in \C: \frac34<\Re(s)<1 \big\}$ with connected complement, and any function $f$ that is continuous on $K$ and analytic in the interior of $K$, we have 
\begin{align*}
\liminf_{T \to \infty} \frac 1 T \mathop{\rm meas} \left \{t \in [0,T]:\max_{s \in
K} \left|E_n\left(L,\frac {n(s+it)}2\right)-f(s)\right|<\varepsilon \right \}>0. 
\end{align*}
\end{thm}

There are a number of standard corollaries of this type of universality (see, e.g., \cite{steuding}). As one such example, we mention an immediate consequence of Theorem \ref{UNIV3} that follows by approximating $f(s)=s-\frac{a+b}2$ on $K=\left\{s \in\C : |s-\frac{a+b}2| \leq \frac{b-a}2\right\}$ and using Rouche's theorem.

\begin{cor}\label{ZEROS2}
Let $n\geq 2$ and let $\frac34<a<b<1$. Then, for almost all $L \in X_n$, we have
\begin{gather*}
\liminf_{T \to \infty} \frac 1 T \# \left \{ \rho \in \C: |\Im(\rho)|<T, a<\Re(\rho)<b, E_n\left(L,\frac{n\rho}{2} \right)=0 \right \}>0.
\end{gather*}
\end{cor}

This corollary complements the recent results in \cite{SodStrom}, that treats zeros in the half-plane $\Re(s)>1$, by showing that also the zeros in the strip $\{s: \frac 34<\Re(s)<1\}$ violates the Riemann hypothesis in a strong way. We also note that the corresponding result for $c$-values (values $s$ such that $E_n\left(L,\frac{ns}{2}\right)=c$) follows by a similar argument.

Finally, we mention that the proof of Theorem \ref{UNIV3} uses a result of Drungilas-Garunk{\v{s}}tis-Ka{\v{c}}{\.{e}}nas \cite{DruGarKac} on universality of general Dirichlet series. The main ingredient in the proof is a new mean square estimate that is valid for $\Re(s)>\frac34$ and  almost all lattices $L \in X_n$ (cf.\ Theorem \ref{MEANSQUARE}). In the process of proving this estimate, we discover a strong bound on the growth of the Epstein zeta function on the critical line, which we find interesting in its own right.

\begin{thm}\label{SUBCONVEXITYTHEOREM}
Let $n\geq2$ and let $\ve>0$. Then, for almost all $L\in X_n$, we have 
\begin{equation}\label{SUBCONVEX}
E_n\left(L,\frac n4+it\right)=O_{L,\ve}\left((1+|t|)^{1+\ve}\right).
\end{equation}
\end{thm}

Recalling that the corresponding convexity bound states that $E_n\left(L,\frac n4+it\right)=O_{L,\ve}((1+|t|)^{\frac n4+\ve})$ for all $L\in X_n$, it follows that \eqref{SUBCONVEX} provides a subconvex estimate for almost all $L\in X_n$ as soon as $n\geq 5$. In fact, the bound \eqref{SUBCONVEX} is majorized by any positive power $\delta$ of the convexity bound, in all sufficiently large dimensions $n$ (depending on $\delta$).

Theorem \ref{SUBCONVEXITYTHEOREM} has recently been improved by Blomer \cite{blomer}; using the spectral theory of automorphic forms, Blomer arrives at the estimate 
\begin{align*}
E_n\left(L,\frac n4+it\right)=O_{L,\ve}\big((1+|t|)^{\frac12+\ve}\big)
\end{align*}
for almost all $L\in X_n$. (A detailed proof of this result is provided only for almost all orthogonal lattices). See also \cite{spinu} for the even stronger bound $E_2\left(L,\frac 12+it\right)=O_{L,\ve}((1+|t|)^{\frac14+\ve})$ valid for almost all $L\in X_2$.

\section{Preliminaries}\label{prelimsec}

\subsection{Poisson distribution of vector lengths}\label{Poissonsection} 

Recall that we use $V_n$ to denote the volume of the unit ball in $\R^n$. Given a lattice $L\in X_n$, we order its nonzero vectors by increasing lengths as $\pm\vecv_1,\pm\vecv_2,\pm\vecv_3,\ldots$, set $\ell_j=|\vecv_j|$ (thus $0<\ell_1\leq \ell_2\leq \ell_3\leq\ldots$), and define
\begin{align}\label{volumes}
 \mathcal V_j(L):=\tfrac12 V_n\ell_j^n\,,
\end{align}
so that $\mathcal V_j(L)$ equals one half of the volume of an $n$-dimensional ball of radius $\ell_j$. The main result in \cite{poisson} states that the volumes $\{\mathcal V_j(L)\}_{j=1}^{\infty}$ determined by a random lattice $L\in X_n$ converges in distribution, as $n\to\infty$, to the points of a Poisson process on the positive real line with constant intensity $1$. In other words, if we for $t\geq0$ let
\begin{align*}
N_{n,L}(t):=\#\left\{j:\mathcal{V}_j(L)\leq t\right\}, 
\end{align*}
then we have the following theorem.

\begin{thm}\label{POISSON}
Let $\mathcal P=\{\mathcal N(t),t\geq0\}$ be a Poisson process on the positive real line with intensity $1$. Then the stochastic process $\{N_{n,L}(t),t\geq0\}$ converges weakly to $\mathcal P$ as $n\to\infty$.
\end{thm}
 
Given $L\in X_n$ and $t\geq0$, we define 
\begin{equation}\label{Rndef}
R_{n,L}(t):=N_{n,L}(t)-t.
\end{equation} 
Note that $1+2R_{n,L}(t/2)$ equals the error term in the circle problem generalized to an $n$-dimensional ball of volume $t$ and a general lattice $L\in X_n$. We recall the following bound on $R_{n,L}(t)$ and refer to \cite[Theorem 1.3]{epstein2} (see also \cite{schmidt}) for a proof.

\begin{thm}\label{Rnthm}
For all $\ve>0$ there exists $C_{\ve}>0$ such that for all $n\geq3$ and $C\geq1$ we have 
\begin{align}\label{RnEST}
{\rm Prob}_{\mu_n}\Big\{L\in X_n : |R_{n,L}(t)|\leq C_{\ve}(Ct)^{\frac{1}{2}}(\log t)^{\frac{3}{2}+\ve},\quad \forall t\geq5\Big\}\geq1-C^{-1}.
\end{align}
\end{thm}

In our discussion, we will often find it convenient to work with a close relative of $R_{n,L}(t)$, namely
\begin{equation}\label{Sn}
S_{n,L}(t):=N_{n,L}(t)-\#\left\{j\in\Z^+ : j\leq t\right\}\qquad\,\, (t\geq0).
\end{equation}
The following estimate is an immediate corollary of Theorem \ref{Rnthm}. 

\begin{cor}\label{SnCOR}
For all $\ve>0$ there exists $C_{\ve}>0$ such that for all $n\geq3$ and $C\geq1$ we have 
\begin{align*}
{\rm Prob}_{\mu_n}\Big\{L\in X_n : |S_{n,L}(t)|\leq C_{\ve}(Ct)^{\frac{1}{2}}(\log t)^{\frac{3}{2}+\ve},\quad \forall t\geq5\Big\}\geq1-C^{-1}.
\end{align*}
\end{cor}

\begin{remark}\label{SCHMIDTREMARK}
Let us note that in the case $n=2$, Schmidt \cite[Theorem 2]{schmidt} has proved that for almost all $L\in X_2$ and for all sufficiently large $t$ (depending on $L$ and $\ve$), we have
\begin{equation*}
|R_{2,L}(t)|\ll_{L,\ve}t^{\frac{1}{2}}(\log t)^{\frac{5}{2}+\ve}.
\end{equation*}
Hence, for the same $L$ and $t$, we also have
\begin{equation*}
|S_{2,L}(t)|\ll_{L,\ve}t^{\frac{1}{2}}(\log t)^{\frac{5}{2}+\ve}.
\end{equation*}
\end{remark}

\subsection{Normalization of $E_n(L,s)$}

For any complex number $s$ with $\Re(s)>1$, we can use \eqref{volumes} to write $E_n\left(L,\frac{ns}{2}\right)$ in the form
\begin{align*}
E_{n}\left(L,\sfrac{ns}{2}\right)=2^{1-s}V_n^{s}\sum_{j=1}^{\infty}\mathcal{V}_j(L)^{-s}.
\end{align*}
We find it natural to consider the normalized function
\begin{align*}
\mathcal{E}_{n}(L,s):=2^{s-1}V_n^{-s}E_{n}\left(L,\sfrac {ns}{2} \right),
\end{align*}
so that 
\begin{align*}
\mathcal{E}_{n}(L,s)=\sum_{j=1}^{\infty}\mathcal V_j(L)^{-s}\qquad\,\,(\Re(s)>1).
\end{align*}
Note that $\mathcal{E}_{n}(L,s)$ has a simple pole at $s=1$ with residue $1$. Note also that if $\{T_j\}_{j=1}^{\infty}$ are the points of a Poisson process on the positive real line with intensity $1$, then \cite[Theorem 1 and Remark 4]{epstein1} state that $\mathcal{E}_{n}(L,s)$ converges in distribution to $\sum_{j=1}^{\infty}T_j^{-s}$, for any fixed $s$ with $\Re(s)>1$, as $n\to\infty$.

\section{Proof of Theorem \ref{UNIV1}}\label{Proof1}

To begin with we state the following lemma of Mishou and Nagoshi (cf.\ \cite[Proposition 2.4]{MiNa1} for a more general statement). We recall that $D$ denotes the right half of the critical strip (see \eqref{Ddef}).

\begin{lem}\label{APPROX1}
Let $\Omega$ be a simply connected region in $D$ that is symmetric with respect to the real axis. Then, for any $\ve>0$, any compact set $K\subset\Omega$ and any holomorphic function $f$ on $\Omega$ that is real-valued on $\Omega\cap\R$, there exist $N_0\in\Z^+$ and numbers $a_j\in\{-1,0,1\}$, $1\leq j\leq N_0$, such that
\begin{align*}
\max_{s\in K}\Bigg|\sum_{j=1}^{N_0}a_jj^{-s}-f(s)\Bigg|<\ve.
\end{align*}
\end{lem}

Given $N\in\Z^+$, $L\in X_n$ and $s\in\C$, we define
\begin{align*}
P_N(L,s):=\sum_{\mathcal V_j(L)\leq N}\mathcal V_j(L)^{-s}-\sum_{j=1}^{N}j^{-s}+\zeta(s).
\end{align*}
Clearly $P_N$ is analytic in $\C$ except for a simple pole at $s=1$ with residue $1$. 

\begin{remark}
In the proof of Theorem \ref{UNIV1} we find it convenient to subtract and add $\zeta(s)$ to $\mathcal{E}_{n}(L,s)$ (see \eqref{elementarysplitting}). The main reason is that $\mathcal{E}_{n}(L,s)-\zeta(s)$ is an entire function and its generalized Dirichlet series representation converges in $\Re s>\frac12$ for almost every $L\in X_n$ (cf.\ the discussion in and just above Lemma \ref{TAILLEMMA} below). On the other hand, the second copy of $\zeta(s)$ plays no essential role in this investigation. We note that $P_N(L,s)$ is simply the sum of $\zeta(s)$ and a suitable truncation of $\mathcal{E}_{n}(L,s)-\zeta(s)$ that is easy to work with in connection with the approximation of Dirichlet polynomials (see the proof of Proposition \ref{KEYLEMMA} below).
\end{remark}

The following approximation result is the key technical ingredient in the proof of Theorem \ref{UNIV1}.   

\begin{prop}\label{KEYLEMMA}
Let $\Omega$ be a simply connected region in $D$ that is symmetric with respect to the real axis. Then, for any $\ve>0$, any compact set $K\subset\Omega$ and any holomorphic function $f$ on $\Omega$ that is real-valued on $\Omega\cap\R$, there exist $N\in\Z^+$ and a constant $\delta>0$ such that
\begin{align}\label{MAINTERM}
\liminf_{n\to\infty}{\rm Prob}_{\mu_n}\left\{L\in X_n : \max _{s\in K}\left|P_M(L,s)-f(s)\right|<\ve\right\}>\delta
\end{align}
holds for any fixed $M\geq N$.
\end{prop}

\begin{proof}
Let $h(s):=\zeta(s)-f(s)$ and note that $h$ is holomorphic on $\Omega$ and real-valued on $\Omega\cap\R$. By Lemma \ref{APPROX1}, there exist $N_0\in\Z^+$ and coefficients $a_j\in\{-1,0,1\}$, $1\leq j\leq N_0$, such that
\begin{align}\label{STEP1}
\max_{s\in K}\Bigg|\sum_{j=1}^{N_0}a_jj^{-s}-h(s)\Bigg|<\frac{\ve}{3}\,.
\end{align}
We let $N>\max(N_0,5)$ be an integer and set
\begin{align*}
b_j=\begin{cases}
1-a_j&\text{if $1\leq j\leq N_0$,}\\
1&\text{if $N_0<j\leq N$.}
\end{cases}
\end{align*}
Hence
\begin{align}\label{STEP2}
P_N(L,s)-f(s)=\sum_{\mathcal V_j(L)\leq N}\mathcal V_j(L)^{-s}-\sum_{j=1}^Nb_jj^{-s}+\Bigg(h(s)-\sum_{j=1}^{N_0}a_jj^{-s}\Bigg).
\end{align}
In addition it is useful to note that  
\begin{equation*}
\sum_{j=1}^Nb_j j^{-s}=\sum_{j=1}^{\mathcal N}n_j^{-s},
\end{equation*}
for some $N-N_0\leq \mathcal N\leq N+N_0$ and a certain nondecreasing sequence $\{n_j\}_{j=1}^{\mathcal N}$ of positive integers. Note in particular that $n_{\mathcal N}=N$.\footnote{It is in principal straightforward, but notationally impractical, to write down explicit formulas for all the integers $n_j$.}

Next, we find that there exist $N_1\in\Z^+$  and $\delta_0>0$ (depending on $K$, $N$ and $\ve$) such that for $n\geq N_1$, we have 
\begin{align}\label{STEP3} 
{\rm Prob}_{\mu_n}\left\{L\in X_n :\max_{s\in K}\Bigg|\sum_{\mathcal V_j(L)\leq N}\mathcal V_j(L)^{-s}-\sum_{j=1}^{\mathcal N}n_j^{-s}\Bigg|<\frac{\ve}{3}\right\}>\delta_0\,.
\end{align}
We recall that $n_{\mathcal N}=N$ and define $n_{\mathcal N+1}:=N+1$. It is clear that \eqref{STEP3} follows if, for a sufficiently small $0<\theta<1$ (depending on $K$, $N$ and $\ve$),
\begin{equation*}
n_j-\theta <\mathcal V_j(L) < n_j \qquad (j=1,\ldots,\mathcal N+1)
\end{equation*}
holds with probability $>\delta_0$ whenever $n \geq N_1$.
Hence, it is sufficient to prove that
\begin{gather}\label{THETA1} 
\mathcal V_1(L)  \in I_1:=\left (n_1-\theta, n_1-\frac {\mathcal N \theta} {\mathcal N+1} \right),\\ 
\intertext{and that}
\label{THETA2}
\mathcal V_{j+1}(L)-\mathcal V_{j}(L)\in I_{j+1}:=\left(n_{j+1}-n_j, n_{j+1}-n_j+ \frac{ \theta} {\mathcal N+1}\right)  
\end{gather}
holds for $1 \leq j \leq \mathcal N$, with the given probability for all large enough $n$. To prove this, we first note that it follows from Theorem \ref{POISSON}  and \cite[Section 4.1]{kingman} that $\{\mathcal V_1(L)\}\cup\{\mathcal V_{j+1}(L)-\mathcal V_j(L)\}_{j=1}^{\mathcal N}$ tend in distribution, as $n\to\infty$, to a collection $\{Y_j\}_{j=1}^{\mathcal N+1}$ of independent exponentially distributed random variables of mean $1$. The probability that $(Y_1,\ldots,Y_{\mathcal N+1})$ lies in the open set $\prod_{j=1}^{\mathcal N+1} I_j$ may be explicitly calculated as
\begin{gather*}
  \delta_0^*:=\left(\exp \left(\frac \theta {\mathcal N+1} \right)-1 \right)^{\mathcal N+1} e^{-N-1}>0.
\end{gather*}
Hence, it follows from \cite[Theorem 2.1]{billconv} that the lower limit as $n\to \infty$ of the probability  that \eqref{THETA1} and \eqref{THETA2} holds for  $1 \leq j \leq {\mathcal N}$ is greater than or equal to $\delta_0^*$. We conclude that \eqref{STEP3} holds for any  $\delta_0^*>\delta_0>0$ and all $n \geq N_1$ (where $N_1$ depends on $\delta_0$).

Note that in the special case where $M=N$ and $\delta=\delta_0$, the inequality \eqref{MAINTERM} follows from \eqref{STEP1}, \eqref{STEP2}, \eqref{STEP3} and an application of the triangle inequality. It is also clear from the argument above that we can increase $N$ as much as we like at the cost of having a possibly smaller constant $\delta_0$. It remains to show that  \eqref{MAINTERM} holds with the same right-hand side $\delta$ for any fixed $M\geq N$. Hence, for $M\geq N$, we study the finite sums
\begin{align*}
Q_{N,M}(L,s):=&P_M(L,s)-P_N(L,s)\\
=&\sum_{N<\mathcal V_j(L)\leq M}\mathcal V_j(L)^{-s}-\sum_{j=N+1}^Mj^{-s}=\int_{N}^{M}t^{-s}\,dS_{n,L}(t),
\end{align*} 
where $S_{n,L}(t)$ is defined in \eqref{Sn}. Let $\sigma_0=\min\left\{\Re(s) : s\in K\right\}$ and $\eta=(\sigma_0-\frac12)/2$. It follows from Corollary \ref{SnCOR} with $C=2$ that for each $n\geq3$ there exists a set $Y_n\subset X_n$ with $\mu_n(Y_n)\geq\frac12$ and such that for all lattices $L\in Y_n$ and all $t\geq5$ we have $|S_{n,L}(t)|\ll_{\eta}t^{\frac12+\eta}$. Here it is important to note that the implied constant is independent of both $n$ and $L$. Integrating by parts we get, for any $L\in Y_n$ and all $M\geq N$,
\begin{align*} 
\left|Q_{N,M}(L,s)\right|=&\left|\int_{N}^{M}t^{-s}\,dS_{n,L}(t)\right|\\
=&\left|\left[t^{-s}S_{n,L}(t)\right]_N^{M}+s\int_N^{M}t^{-s-1}S_{n,L}(t)\,dt\right|
\ll_{\eta,K}N^{\frac12+\eta-\sigma_0},\nonumber
\end{align*} 
uniformly for all $s\in K$. Since we can make the right-hand side above as small as we desire by increasing $N$ as needed, we obtain, for any $M\geq N$ and $n\geq 3$,
\begin{align}\label{STEP4}
{\rm Prob}_{\mu_n}\left\{L\in X_n :\max_{s\in K}\left|Q_{N,M}(L,s)\right|<\frac{\ve}{3}\right\}\geq\frac12\,.
\end{align}

Finally, given $M\geq N$, we would like to use \eqref{STEP1}, \eqref{STEP2}, \eqref{STEP3} and \eqref{STEP4} to conclude that \eqref{MAINTERM} holds with (say) $\delta=\delta_0/4$. In order to verify that this is possible, we recall that Theorem \ref{POISSON} states that the volumes $\{\mathcal V_{j}(L)\}_{j=1}^{\infty}$ tend in distribution (as $n\to\infty$) to the points of a Poisson process $\mathcal P$ on the positive real line of constant intensity $1$. Note that the proof of \eqref{STEP3} only uses the restriction of these processes to the open interval $(0,N)$ and similarly that the proof of \eqref{STEP4} only uses the restriction of these processes to $(N,\infty)$. Now the crucial observation is that the process $\mathcal P$ may be realized as a union of a Poisson process on $(0,N)$ and an \emph{independent} Poisson process on $(N,\infty)$, both of intensity $1$ (see, e.g., \cite[Section 2.2]{kingman}). Hence it follows from \eqref{STEP1}, \eqref{STEP2}, \eqref{STEP3} and \eqref{STEP4}, that when $n\geq \max(N_1,3)$ (where we might need to increase $N_1$ depending on $M$) we have 
\begin{align}\label{FINALSTEP} 
{\rm Prob}_{\mu_n}\left\{L\in X_n :\max_{s\in K}\left|P_M(L,s)-f(s)\right|<\ve\right\}>\frac{\delta_0}{4}\,.
\end{align}
In fact, as $n\to\infty$ we have the above inequality with any right-hand side strictly smaller than $\frac{\delta_0}{2}$. Since the lower bound in \eqref{FINALSTEP} holds for any fixed $M\geq N$, the proof is complete.
\end{proof}

Next we define, for $N\in\Z^+$ and $\Re(s)>1$, the function
\begin{align*}
Q_N(L,s):=\sum_{\mathcal V_j(L)>N}\mathcal V_j(L)^{-s}-\sum_{j=N+1}^{\infty}j^{-s}=\int_N^{\infty}t^{-s}\,dS_{n,L}(t),
\end{align*}
and note that this integral representation is holomorphic in the half-plane $\Re(s)>\frac12$ for almost every $L\in X_n$.

\begin{lem}\label{TAILLEMMA}
Let $n\geq 3$. Let $\ve,\delta>0$ and let $K$ be a compact subset of the half-plane $\Re(s)>\frac12$. Then there exists $M\in\Z^+$ such that 
\begin{align*}
{\rm Prob}_{\mu_n}\left\{L\in X_n : \max _{s\in K}\bigg|\int_{M'}^{\infty}t^{-s}\,dS_{n,L}(t)\bigg|<\ve\right\}\geq1-\delta\,
\end{align*}
for all $M'\geq M$. 
\end{lem}

\begin{proof}
Let $\delta>0$ be given. Recall from the proof of Proposition \ref{KEYLEMMA} that $\sigma_0=\min\left\{\Re(s) : s\in K\right\}$ and $\eta=(\sigma_0-\frac12)/2$. It follows from Corollary \ref{SnCOR} that for each $n\geq3$ there exists a set $Z_n\subset X_n$ with $\mu_n(Z_n)\geq1-\delta$ and such that for all lattices $L\in Z_n$ and all $t\geq5$ we have $|S_{n,L}(t)|\ll_{\delta,\eta}t^{\frac12+\eta}$, where the implied constant is independent of $n$ and $L$. Now, for any $L\in Z_n$ and all $M\geq5$, we get
\begin{align*}
\left|\int_{M}^{\infty}t^{-s}\,dS_{n,L}(t)\right|=\left|\left[t^{-s}S_{n,L}(t)\right]_M^{\infty}+s\int_M^{\infty}t^{-s-1}S_{n,L}(t)\,dt\right|
\ll_{\delta,\eta,K}M^{\frac12+\eta-\sigma_0},
\end{align*} 
uniformly for all $s\in K$. Since we can make the right-hand side above as small as desired by choosing $M$ large enough, the proof is complete.
\end{proof}

\begin{proof}[Proof of Theorem \ref{UNIV1}]
Let $\ve>0$ be given and let $N$ and $\delta$ be given by Proposition \ref{KEYLEMMA} applied with $\ve/2$, $K$ and $f$. Let furthermore $M$ be given by Lemma \ref{TAILLEMMA} applied with $\ve/2$, $\delta/2$ and $K$. (Note that we without loss of generality may assume that $M\geq N$.) Hence, by Proposition \ref{KEYLEMMA} and Lemma \ref{TAILLEMMA}, we find that
\begin{multline*}
\left|\mathcal{E}_{n}(L,s)-f(s)\right|=\left|\mathcal{E}_{n}(L,s)-\zeta(s)+\zeta(s)-f(s)\right|\\
\leq\left|P_{M}(L,s)-f(s)\right|+\left|Q_{M}(L,s)\right|<\frac{\ve}{2}+\frac{\ve}{2}=\ve
\end{multline*}
holds, in the limit as $n\to\infty$, with a probability of at least $\delta/2$. This finishes the proof. 
\end{proof}

\section{Proof of Theorem \ref{UNIV2}}

The general strategy for the proof of Theorem \ref{UNIV2} is the same as in the proof of Theorem \ref{UNIV1}. The main difference is that Lemma \ref{APPROX1} is no longer at our disposal. The following approximation lemma will take its place in the present argument.

\begin{lem}\label{APPROX2} 
Let $\Omega$ be a simply connected region in $\C$ that is
symmetric with respect to the real axis. Then, for any $\varepsilon>0$, any
compact set $K\subset\Omega$ and any holomorphic function $f$ on $\Omega$
that is real-valued on $\Omega\cap\R$, there exists a Dirichlet polynomial 
\begin{gather*} 
A(s)=\sum_{k=1}^{N_0} a_k \lambda_k^{-s}  
\\ \intertext{with integer coefficients such that}
 \max_{s \in K} \left|A(s)-f(s)\right|<\varepsilon. 
 \end{gather*}
\end{lem}

\begin{proof}
The conditions on $f$ and $\Omega$ are sufficient to ensure that
\begin{gather} \label{eqfkonjugat}
 f(s)=\overline{f(\overline s)} \qquad (s \in \Omega).
\end{gather}
Let $K^\sharp=K \cup \overline K$, where $\overline K$ is the reflection of $K$ in the
real axis. By Runge's theorem (applied to a simply connected compact set $K'$ satisfying $K^\sharp\subset K'\subset \Omega$) there exists a polynomial $P(s)$ such that
\[ |P(s)-f(s)|<\frac{\varepsilon} 2 \qquad (s \in K^\sharp). \]
By the symmetry of $K^\sharp$, we have 
\[ \big|\overline{P(\overline s)}-\overline{f(\overline
s)}\big|<\frac{\varepsilon} 2 \qquad (s \in K^\sharp), \]
and by \eqref{eqfkonjugat} and the triangle inequality it follows that
\begin{gather} \label{Pdef}  |p(s)-f(s)|<\frac{\varepsilon} 2 \qquad (s 
\in K), 
\end{gather}
where
\[p(s)=\frac{P(s)+\overline{P(\overline s)}} 2 =
\frac{P(s)+\overline{P}(s)} 2 \]
is a polynomial with real coefficients. 

We now proceed to approximate the polynomial $p(s)$ by Dirichlet polynomials of the desired type. 
It is sufficient to prove that \[p_m(s)=b_m s(s-1) \cdots (s-m+1) \] may
be approximated by a Dirichlet polynomial $A_m(s)$, since $p(s)$
can be written as a sum of such polynomials, that is,
\begin{equation*}
p(s)=\sum_{m=0}^M p_m(s) 
\end{equation*}
for some nonnegative integer $M$ and suitable real numbers $\{b_m\}_{m=0}^M$. Let 
\begin{equation*}
g(x)=b_m x^s. 
\end{equation*} 
It is clear that 
\begin{equation*}
g^{(m)}(x)= b_m s (s-1) \cdots (s-m+1) x^{s-m}.
\end{equation*} 
Note that this derivative can be approximated up to any given accuracy by $\Delta_h^m g(x)$, where
\begin{equation*}
\Delta_h g(x) := \frac{g(x+h)-g(x)}{h},
\end{equation*} 
by choosing $h$ small enough. Note also that
\begin{equation*}
\Delta_h^m g(x)=\frac{b_m}{h^m}\sum_{j=0}^m(-1)^{m-j}\binom{m}{j}\left(x+jh\right)^s.
\end{equation*}
It follows that any such $\Delta_h^m g(x)$, where we
choose $x=1$ and $h$ such that $b_m/h^m$ is an integer, will
be of the desired type and can be chosen as our Dirichlet polynomial
$A_m(s)$. Thus, by choosing $h$ sufficiently small, we can ensure that 
 \begin{gather} \label{Aolikhet}
    \left| A_m(s) -p_m(s) \right| < \frac {\varepsilon}{2(M+1)} \qquad (0\leq m\leq M, s
\in K).
\end{gather}
 The result follows by choosing 
\[A(s)=\sum_{m=0}^M A_m(s) \]
and applying the triangle inequality together with the inequalities \eqref{Pdef} and
\eqref{Aolikhet}.
 \end{proof}

Given $\Lambda\in\R^+$, two lattices $L_1,L_2\in X_n$ and a complex number $s\in\C$, we define
\begin{align*}
\widetilde P_{\Lambda}(L_1,L_2,s):=\sum_{\mathcal V_j(L_1)\leq \Lambda}\mathcal V_j(L_1)^{-s}-\sum_{\mathcal V_j(L_2)\leq \Lambda}\mathcal V_j(L_2)^{-s},
\end{align*}
which clearly is an entire function of $s$.

\begin{prop}\label{KEYLEMMA2}
Let $\Omega$ be a simply connected region in the half-plane $\Re(s)>\frac12$ that is symmetric with respect to the real axis. Then, for any $\ve>0$, any compact set $K\subset\Omega$ and any holomorphic function $f$ on $\Omega$ that is real-valued on $\Omega\cap\R$, there exist constants $\delta,\Lambda_0\in\R^+$ such that
\begin{align}\label{MAINTERM2}
\liminf_{n\to\infty}{\rm Prob}_{\mu_n\times \mu_n}\left\{(L_1,L_2)\in X_n^2 : \max _{s\in K}\left|\widetilde P_{\Lambda}(L_1,L_2,s)-f(s)\right|<\ve\right\}>\delta
\end{align}
holds for any fixed $\Lambda\geq\Lambda_0$.
\end{prop}

\begin{proof}
By Lemma \ref{APPROX2}, there exist an integer $N_0\geq1$ and sequences $\{\lambda_k\}_{k=1}^{N_0}$ and $\{a_k\}_{k=1}^{N_0}$ of positive real numbers and integers respectively, such that
\begin{align}\label{STEP1B}
\max_{s\in K}\Bigg|\sum_{k=1}^{N_0}a_k\lambda_k^{-s}-f(s)\Bigg|<\frac{\ve}{3}\,.
\end{align}
Let $\lambda_{0}=5$, $C_1=\sum_{a_k>0}a_k$ and $C_2=\sum_{a_k<0}|a_k|$. By choosing $\Lambda_0\geq\max\{\lambda_k : 0\leq k\leq N_0\}$ sufficiently large, we find that there exist $N_1\in\Z^+$  and $\delta_0>0$ (depending on $K$, $\Lambda_0$ and $\ve$) such that for $n\geq N_1$, we have 
\begin{align}\label{STEP3B} 
{\rm Prob}_{\mu_n\times \mu_n}\left\{(L_1,L_2)\in X_n^2 :\max_{s\in K}\Bigg|\widetilde P_{\Lambda_0}(L_1,L_2,s)-\sum_{k=1}^{N_0}a_k \lambda_k^{-s}\Bigg|<\frac{\ve}{3}\right\}>\delta_0\,.
\end{align}
The proof of \eqref{STEP3B} is essentially the same as the proof of \eqref{STEP3} (see the proof of Proposition \ref{KEYLEMMA}). The only difference is that we in this case use the sequence $\{\mathcal V_{j}(L_1)\}_{j=1}^{C_1}$ to approximate the numbers $\lambda_k$ with positive coefficients $a_k$ and the sequence $\{\mathcal V_{j}(L_2)\}_{j=1}^{C_2}$ to approximate the numbers $\lambda_k$ with negative coefficients $a_k$ (in both cases with multiplicities according to the sizes of the corresponding $|a_k|$). Now, in the case where $\Lambda=\Lambda_0$ and $\delta=\delta_0$, the inequality \eqref{MAINTERM2} follows from \eqref{STEP1B}, \eqref{STEP3B} and an application of the triangle inequality. Note that we may increase $\Lambda_0$ at the cost of having a possibly smaller constant $\delta_0$. 

In order to show that  \eqref{MAINTERM2} holds with the same right-hand side $\delta$ for any fixed $\Lambda\geq\Lambda_0$, we study the finite sums
\begin{align*}
\widetilde Q_{\Lambda_0,\Lambda}(L_1,L_2,s):=&\widetilde P_{\Lambda}(L_1,L_2,s)-\widetilde P_{\Lambda_0}(L_1,L_2,s)\\
=&\sum_{\Lambda_0<\mathcal V_j(L_1)\leq\Lambda}\mathcal V_j(L_1)^{-s}-\sum_{\Lambda_0<\mathcal V_j(L_2)\leq\Lambda}\mathcal V_j(L_2)^{-s}\\
=&\int_{\Lambda_0}^{\Lambda}t^{-s}\,dR_{n,L_1}(t)-\int_{\Lambda_0}^{\Lambda}t^{-s}\,dR_{n,L_2}(t),
\end{align*} 
where $R_{n,L}(t)$ is defined in \eqref{Rndef}. Let $\sigma_0=\min\left\{\Re(s) : s\in K\right\}$ and $\eta=(\sigma_0-\frac12)/2$. It follows from Theorem \ref{Rnthm} with $C=2$ that for each $n\geq3$ there exists a set $U_n\subset X_n$ with $\mu_n(U_n)\geq\frac12$ and such that for all lattices $L\in U_n$ and all $t\geq5$ we have $|R_{n,L}(t)|\ll_{\eta}t^{\frac12+\eta}$. Integrating by parts we get, for all $L_1,L_2\in U_n$ and all $\Lambda\geq\Lambda_0$,
\begin{align}\label{QTILDECALC} 
\left|\widetilde Q_{\Lambda_0,\Lambda}(L_1,L_2,s)\right|\leq&\left|\int_{\Lambda_0}^{\Lambda}t^{-s}\,dR_{n,L_1}(t)\right|+\left|\int_{\Lambda_0}^{\Lambda}t^{-s}\,dR_{n,L_2}(t)\right|\ll_{\eta,K}\Lambda_0^{\frac12+\eta-\sigma_0},
\end{align} 
uniformly for all $s\in K$. Thus, since we can make the right-hand side above as small as we like by a sufficient increase of $\Lambda_0$, we obtain, for any $\Lambda\geq\Lambda_0$ and $n\geq 3$,
\begin{align}\label{STEP4B}
{\rm Prob}_{\mu_n\times\mu_n}\left\{(L_1,L_2)\in X_n^2 : \max_{s\in K}\big|\widetilde Q_{\Lambda_0,\Lambda}(L_1,L_2,s)\big|<\frac{\ve}{3}\right\}\geq\frac14\,.
\end{align}

Finally, given $\Lambda\geq\Lambda_0$, we use \eqref{STEP1B}, \eqref{STEP3B}, \eqref{STEP4B} and Theorem \ref{POISSON} to conclude that when $n\geq \max(N_1,3)$ (where we might need to increase $N_1$ depending on $\Lambda$) we have
\begin{align}\label{FINALSTEPB} 
{\rm Prob}_{\mu_n\times\mu_n}\left\{(L_1,L_2)\in X_n^2 :\max_{s\in K}\left|\widetilde P_{\Lambda}(L_1,L_2,s)-f(s)\right|<\ve\right\}>\frac{\delta_0}{8}\,.
\end{align}
This finishes the proof. 
\end{proof}

We define, for each $\Delta\in\R^+$ and $\Re(s)>1$, the function 
\begin{multline*}
\widetilde Q_{\Delta}(L_1,L_2,s):=\sum_{\mathcal V_j(L_1)>\Delta}\mathcal V_j(L_1)^{-s}-\sum_{\mathcal V_j(L_2)>\Delta}\mathcal V_j(L_2)^{-s}\\
=\int_{\Delta}^{\infty}t^{-s}\,dR_{n,L_1}(t)-\int_{\Delta}^{\infty}t^{-s}\,dR_{n,L_2}(t),
\end{multline*}
and note that this integral representation is holomorphic in the half-plane $\Re(s)>\frac12$ for almost every $(L_1,L_2)\in X_n^2$.

\begin{lem}\label{TAILLEMMA2}
Let $n\geq 3$. Let $\ve,\delta>0$ and let $K$ be a compact subset of the half-plane $\Re(s)>\frac12$. Then there exists $\Lambda\in\R^+$ such that 
\begin{align*}
{\rm Prob}_{\mu_n\times\mu_n}\left\{(L_1,L_2)\in X_n^2 : \max _{s\in K}\left|\int_{\Lambda'}^{\infty}t^{-s}\,dR_{n,L_1}(t)-\int_{\Lambda'}^{\infty}t^{-s}\,dR_{n,L_2}(t)\right|<\ve\right\}\geq1-\delta\,
\end{align*}
for all $\Lambda'\geq\Lambda$.
\end{lem}

\begin{proof}
The result follows from Theorem \ref{Rnthm} applied with $C\geq2/\delta$ to an estimate similar to the one in \eqref{QTILDECALC}.
\end{proof}

\begin{proof}[Proof of Theorem \ref{UNIV2}]
Let $\ve>0$ be given and let $\Lambda_0$ and $\delta$ be given by Proposition \ref{KEYLEMMA2} applied with $\ve/2$, $K$ and $f$. Let furthermore $\Lambda$ be given by Lemma \ref{TAILLEMMA2} applied with $\ve/2$, $\delta/2$ and $K$. (We may assume that $\Lambda\geq\Lambda_0$.) Then, it follows from Proposition \ref{KEYLEMMA2} and Lemma \ref{TAILLEMMA2} that
\begin{align*}
\left|\mathcal{E}_{n}(L_1,s)-\mathcal{E}_{n}(L_2,s)-f(s)\right|
\leq\left|\widetilde P_{\Lambda}(L_1,L_2,s)-f(s)\right|+\left|\widetilde Q_{\Lambda}(L_1,L_2,s)\right|<\ve
\end{align*}
holds, in the limit as $n\to\infty$, with a probability of at least $\delta/2$.
\end{proof}

\section{Proof of the imaginary shift case (Theorem \ref{UNIV3})}\label{IMSHIFTSEC}
 
We will use Drungilas-Garunk{\v{s}}tis-Ka{\v{c}}{\.{e}}nas variant \cite[Theorem 2.1]{DruGarKac} of
a result of Gonek \cite{gonek} for general Dirichlet series. To do this we need a
mean square estimate in a half-plane.\footnote{We note here that Blomer \cite{blomer} recently stated a mean square estimate on the critical line $\frac n4+it$ for almost all $L \in X_n$. Using this bound instead of our Theorem \ref{MEANSQUARE} in the proof of Theorem \ref{UNIV3} would result in a stronger universality result valid in $D=\big\{s \in \C: \frac12<\Re(s)<1 \big\}$.} 
 
\begin{thm}\label{MEANSQUARE}
Let  $n \geq 2$. Then, for $\sigma>\frac34$ and for almost all $L \in X_n$, we have
\begin{equation*}
\int_{1}^T \left|E_n\left(L,\frac{n\left(\sigma+it\right)} 2\right)\right|^2\,dt \ll_{L,\sigma} T.
\end{equation*}
\end{thm}
 
By the definition of the normalized Epstein zeta function, the triangle inequality and the well known mean square estimate for the Riemann zeta function
in the half-plane $\sigma>\frac12$ (see e.g. Ivi{\'c} \cite{ivic}), it is sufficient to prove that
\begin{gather} \label{eps}
\int_1^T \big| \mathcal{E}_n(L,\sigma+it)-\zeta(\sigma+it) \big|^2\,dt \ll_{L,\sigma} T .
\end{gather}
We will prove this estimate using Gallagher's lemma \cite[Lemma 1]{gallagher}.

\begin{lem}[Gallagher]
Let $\delta=\theta/T$, with $0<\theta<1$, and let $$S(t):=\sum_k c_k e(v_k t)$$ be an absolutely convergent exponential sum having only real frequencies $v_k$. Then
\begin{equation*}
\int_{-T}^T \left|S(t)\right|^2\,dt \ll_\theta \int_{-\infty}^\infty \bigg|\delta^{-1} \sum_{x \leq v_k \leq x+\delta}c_k \bigg|^2\,dx.
\end{equation*}
\end{lem}

We will not be able to apply Gallagher's lemma directly, since it assumes that the exponential sum $S(t)$ is absolutely convergent. Thus, we will need to truncate our Dirichlet series; in fact, we will apply the following weak approximate functional equation.

\begin{lem}\label{LEMMAAPPROXFUNC} 
Let $\ve>0$. Let $n \geq 2$ and assume that $\Re(s)>\frac12$. Then, for any $X\geq5$ and almost every $L \in X_n$, we have
\begin{equation*}
\mathcal{E}_n(L,s)-\zeta(s) = \sum_{\mathcal V_j(L)\leq X}\mathcal V_j(L)^{-s}-\sum_{1\leq j\leq X}j^{-s}+ O_{L,\ve}\left(|s| X^{\frac12-\Re(s)+\varepsilon}\right).
\end{equation*}
\end{lem}

\begin{proof}
The lemma follows immediately from Lemma \ref{TAILLEMMA} and its proof (using the estimate in Remark \ref{SCHMIDTREMARK} in place of Corollary \ref{SnCOR} in the case $n=2$).
\end{proof}
 
\begin{proof}[Proof of Theorem \ref{SUBCONVEXITYTHEOREM}]
The theorem is a direct consequence of Lemma \ref{LEMMAAPPROXFUNC}, the functional equation \eqref{FUNCTIONALEQ} and the Phragmén-Lindel\"of principle.
\end{proof} 
 
In order to apply Gallagher's lemma, we also need estimates for short partial sums of $\mathcal{E}_n(L,s)-\zeta(s)$. We define
\begin{align}\label{DeltaDef} 
\Delta(x)&=\Delta_{\sigma,L,X}(x)=\Delta_1(x)-\Delta_2(x),  \\  \intertext{where} \notag  \Delta_{1}(x)&= \Delta_{\sigma,L,X;1}(x)=\sum_{\mathcal V_j(L) \leq \min( e^x,X)} \mathcal V_j(L)^{-\sigma}\\ \intertext{and}
\notag \Delta_2(x) &= \Delta_{\sigma,L,X;2}(x)= \sum_{1 \leq  j \leq  \min( e^x,X)} j^{-\sigma}.  
\end{align}
To be precise, we need the following result.
 
\begin{lem}\label{PARTIALSUMSESTIMATE}
Let $\ve>0$. Let further $n \geq 2$, $\sigma>\frac12$, $0\leq\delta\leq1$ and $\alpha\in\R$. Then, for almost all $L \in X_n$, we have
\begin{gather*}   
\int_{\alpha}^{\alpha+1} \left| \Delta_{\sigma,L,X}(x+\delta)-\Delta_{\sigma,L,X}(x) \right|^2\,dx \ll_{L,\sigma,\varepsilon}  \delta e^{\alpha(\frac32-2 \sigma+\varepsilon)}.
\end{gather*}
\end{lem}

\begin{proof}
To begin, we have
\begin{multline*}
\int_{\alpha}^{\alpha+1}\left| \Delta(x+\delta)-\Delta(x)\right|^2\,dx\\ 
\leq\max_{x \in [\alpha,\alpha+1]}\left|\Delta(x+\delta)-\Delta(x)\right| \int_{\alpha}^{\alpha+1}\left|\Delta(x+\delta)-\Delta(x)\right|\,dx. 
\end{multline*}
By Corollary \ref{SnCOR} (or Remark \ref{SCHMIDTREMARK} in the case $n=2$) and integration by parts, we obtain
\begin{equation*}
\left|\Delta(x+\delta)-\Delta(x)\right|=\left|\int_{\min(e^x,X)}^{\min(e^{x+\delta},X)}t^{-\sigma}\,dS_{n,L}(t)\right|\ll_{L,\sigma,\varepsilon}e^{x(\frac12-\sigma+\varepsilon)}
\end{equation*}
for almost all $L \in X_n$. For the rest of this proof we restrict our attention to the set of lattices satisfying the above estimate, for which we have
\begin{equation}\label{FIRSTALPHAEST}
\int_{\alpha}^{\alpha+1}\left|\Delta(x+\delta)-\Delta(x)\right|^2\,dx\ll_{L,\sigma,\varepsilon}e^{\alpha(\frac12-\sigma+\varepsilon)}  \int_{\alpha}^{\alpha+1}\left|\Delta (x+\delta)-\Delta(x)\right|\,dx. 
\end{equation}
Using \eqref{DeltaDef} and the triangle inequality, we find that the latter integral is bounded by
\begin{equation}\label{DELTAIEST}
\sum_{i=1}^2 \int_{\alpha}^{\alpha+1}\left|\Delta_{i}(x+\delta)-\Delta_{i}( x)\right|\,dx. 
\end{equation}
Since $\Delta_{i}(t)$ is nondecreasing for $i=1,2$, the absolute values in \eqref{DELTAIEST} can be removed and the above sum equals
\begin{equation*} 
\sum_{i=1}^2 \left(\int_{\alpha}^{\alpha+1}\Delta_{i}(x+\delta)\,dx - \int_{\alpha}^{\alpha+1} \Delta_{i}(x)\,dx\right).
\end{equation*}
Changing variables $y=x+\delta$ in the first integral, we obtain 
\begin{equation*}
\sum_{i=1}^2 \left(\int_{\alpha+\delta}^{\alpha+1+\delta}\Delta_{i}(y)\,dy - \int_{\alpha}^{\alpha+1}\Delta_{i}(x)\,dx\right)
=\sum_{i=1}^2 \int_{\alpha}^{\alpha+\delta}\big(\Delta_{i}(x+1)-\Delta_{i}(x)\big)\,dx,
\end{equation*}
which by Corollary \ref{SnCOR} (or Remark \ref{SCHMIDTREMARK} in the case $n=2$) is $O_{L,\sigma}(\delta e^{\alpha(1-\sigma)})$. Using this estimate, together with \eqref{FIRSTALPHAEST}, we obtain the desired result.
\end{proof}
 
\begin{proof}[Proof of Theorem \ref{MEANSQUARE}] 
Let $\sigma>\frac34$ be given. By the triangle inequality, the mean square estimate for the Riemann zeta function and Lemma 5.3 (applied with $X=T^4$) we have, for almost all $L\in X_n$,
\begin{gather}\label{yyyy}
\int_{1}^T \left|E_n\left(L,\frac{n\left(\sigma+it\right)}{2}\right)\right|^2\,dt 
\ll_{L,\sigma}\int_1^T \Bigg| \sum_{\mathcal V_j(L)\leq T^4}\mathcal V_j(L)^{-s}-\sum_{1\leq j\leq T^4}j^{-s}\Bigg|^2\,dt+T.
\end{gather}
The integral in the right-hand side of \eqref{yyyy} can be estimated by an application of Lemma 5.2 with $\theta=\frac 1 2$, resulting in the upper bound
\begin{gather}\label{DELTAINTEGRAL}
O(T^2)\int_{-\infty}^{\infty}\left|\Delta\left(x+\frac{1}{2T}\right)-\Delta(x)\right|^2\,dx,
\end{gather}
where $\Delta(x)=\Delta_{\sigma,L,X}(x)$ is defined in \eqref{DeltaDef}. The integrand has compact support contained in $[A,B]$, where $A=\log(\min(\mathcal V_1(L),1))+O(T^{-1})$ and $B=4 \log T$.
By dividing the integral into parts of length $1$ and using Lemma \ref{PARTIALSUMSESTIMATE} to estimate each term, we find that the expression in \eqref{DELTAINTEGRAL} is bounded by 
\begin{equation*}
O_{L,\sigma,\ve}(T)\sum_{k=\lfloor A \rfloor}^{\lfloor B \rfloor}  e^{k(\frac32-2 \sigma+\varepsilon)},
\end{equation*}
where the geometric sum is bounded with respect to $T$ whenever we choose $0<\ve<2\sigma-\frac32$. This finishes the proof.
\end{proof}

Before we turn to the proof of Theorem \ref{UNIV3}, we introduce some more notation. We let $\widehat L$ denote a set containing one representative from each pair $\{\vecm,-\vecm\}$
of primitive vectors\footnote{A primitive vector (in a lattice $L$) is a non-zero lattice vector which is not a positive integral multiple of another lattice vector.} in $L$, and write $\{\widehat{\mathcal V}_j(L)\}_{j=1}^{\infty}$ for the subsequence of $\{{\mathcal V}_j(L)\}_{j=1}^{\infty}$ corresponding to vectors $\vecv\in\widehat L$. We also let 
\begin{equation*}
\widehat N_{n,L}(t):=\#\left\{j : \widehat{\mathcal{V}}_j(L)\leq t\right\} \qquad (t\geq0),
\end{equation*}
and note that 
\begin{equation}\label{N-Nhat-relation}
N_{n,L}(t)=\sum_{1\leq m\leq (t/{\mathcal V}_1(L))^{1/n}}\widehat N_{n,L}\Big(\frac t{m^n}\Big).
\end{equation}
We will need a bound on the error term in the primitive circle problem, generalized to an $n$-dimensional ball of volume $t$ and a generic lattice $L\in X_n$, of the same quality as the bounds in Theorem \ref{Rnthm} and Remark \ref{SCHMIDTREMARK}.

\begin{lem}
Let $\ve>0$ and $n\geq2$. Then, for almost all $L\in X_n$, we have
\begin{align}\label{primitivecount}
\widehat N_{n,L}(t)=\frac{t}{\zeta(n)}+O_{L,\ve}\big(t^{\frac12+\ve}\big).
\end{align}
\end{lem}

\begin{proof}
We begin by defining (for $r\geq0$)
\begin{equation*}
\mathcal N_{n,L}(r):=\frac12\#\left\{\vecm\in L\setminus\{\vec0\}: |\vecm|\leq\ell_1 r\right\} 
\end{equation*}
and 
\begin{equation*}
\widehat{\mathcal{N}}_{n,L}(r):=\#\left\{\vecm\in \widehat L: |\vecm|\leq\ell_1 r\right\},
\end{equation*}
where $\ell_1$ denotes the length of a shortest nonzero vector in $L$. We note that these counting functions are related to the ones in \eqref{N-Nhat-relation} by
\begin{equation*}
N_{n,L}(t)=\mathcal N_{n,L}\left(\left(\frac{t}{{\mathcal V}_1(L)}\right)^{\frac1n}\right)
\end{equation*}
and
\begin{equation*}
\widehat N_{n,L}(t)=\widehat{\mathcal{N}}_{n,L}\left(\left(\frac{t}{{\mathcal V}_1(L)}\right)^{\frac1n}\right).
\end{equation*}
Hence, writing $x=x(t,L)= (t/{\mathcal V}_1(L))^{1/n}$, we can reformulate \eqref{N-Nhat-relation} as 
\begin{equation*}
\mathcal N_{n,L}(x)=\sum_{1\leq m\leq x}\widehat{\mathcal{N}}_{n,L}\left(\frac xm\right).
\end{equation*}
Now, using the M\"obius inversion formula in the form given in \cite[Theorem 268]{HW} together with Theorem \ref{Rnthm} and Remark \ref{SCHMIDTREMARK}, we obtain
\begin{equation*}
\widehat{\mathcal{N}}_{n,L}(x)=\sum_{1\leq m\leq x}\mu(m)\mathcal N_{n,L}\left(\frac xm\right)=\sum_{1\leq m\leq x}\mu(m)\left(\frac{t}{m^n}+O_{L,\ve}\left(\frac{t}{m^n}\right)^{\frac12+\ve}\right)
\end{equation*}
for almost every $L\in X_n$. Extending the upper bound in the summation of the main terms to infinity and then estimating the error terms trivially, we arrive at the desired result.
\end{proof}

We further define
\begin{equation*}
\widehat{\mathcal E}_{n}(L,s):=\widehat{\mathcal V}_1(L)^{s}\sum_{\vecm\in \widehat L}\left(\tfrac12 V_n|\vecm|^{n}\right)^{-s}= 1+ \sum_{j=2}^\infty \left( \frac{\widehat{\mathcal V}_j(L)}{\widehat{\mathcal V}_1(L)} \right)^{-s} \qquad (\Re(s)>1) 
\end{equation*} 
and note that 
\begin{align}\label{PRIMITIVEEPSTEIN}
E_{n}\left(L,\tfrac{ns}{2}\right)=2\zeta(ns)\bigg(\frac{V_n}{2\widehat{\mathcal V}_1(L)}\bigg)^s\widehat{\mathcal E}_{n}(L,s).
\end{align}

\begin{proof}[Proof of Theorem \ref{UNIV3}] 
To begin, we note that since $\sigma_0=\min\{ \Re(s) : s \in K\}>\frac34>\frac12$, we can for any $\ve_1>0$ choose $N_0\in\Z^+$ so that for $n \geq 2$
\begin{align}\label{Truncatedzeta}
\max_{\Re(s) \geq \sigma_0} \left|\prod_{\substack {p < N_0}}  (1-p^{-ns})-\frac1{\zeta(ns)} \right| < \ve_1.
\end{align} 
Furthermore, it follows from \cite[Lemma 2]{SodStrom} that 
the set \{$\log \widehat{\mathcal V}_j(L)-\log \widehat{\mathcal V}_1(L): j \geq 2 \}$  
is linearly independent over $\Q$ for almost all lattices $L\in X_n$. From now on we restrict our attention to lattices $L\in X_n$ with this property. We let
$$\Gamma=\left\{\frac{\widehat{\mathcal V}_j(L)}{\widehat{\mathcal V}_1(L)} : j\geq2\right\} \cup\left\{p^n : p<N_0\right\} \cup \left \{\frac{V_n}{2\widehat{\mathcal V}_1(L)} \right \}$$ and choose $\gammamax\subseteq\Gamma$ to be maximal with respect to the property of having the logarithms of all its elements linearly independent over $\Q$. Since the differences $\log \widehat{\mathcal V}_j(L)-\log \widehat{\mathcal V}_1(L)$ for $j \geq 2$ are linearly independent over $\mathbb Q$ by assumption, we will always include these elements in $\gammamax$. We write $$\gammamax=\{\lambda_k\}_{k=1}^{\infty}.$$ Note that the logarithm of each element $\lambda$ in the finite set $\Gamma\setminus\{\widehat{\mathcal V}_j(L)/\widehat{\mathcal V}_1(L)\}_{j=2}^{\infty}$ can be written as
\begin{gather} \label{abb} 
\log \lambda = \sum_{k=1}^{N_2} q_{k,\lambda} \log{\lambda_k}= \sum_{k=1}^{N_2} \frac{h_{k,\lambda}}{N_1} \log{\lambda_k}
\end{gather}
for $\lambda_k \in \gammamax$, $q_{k,\lambda}\in\Q$, $h_{k,\lambda}\in\Z$ and $N_1,N_2\in\Z^+$ where $N_1$ can be chosen as the least common multiple of the denominators of the rational numbers $q_{k,\lambda}$ (for all $k$ and $\lambda$). Note in particular that $N_1$ and $N_2$ are chosen independently of $\lambda$. This is possible since $\Gamma\setminus\{\widehat{\mathcal V}_j(L)/\widehat{\mathcal V}_1(L)\}_{j=2}^{\infty}$ is finite and the integers $h_{k,\lambda}$ are allowed to be zero. We also find it convenient to introduce
$$N_3:=\max_\lambda \sum_{k=1}^{N_2} |h_{k,\lambda}|.$$ 

Next, we define the Dirichlet series
$$A(s):=\sum_{k=1}^{N_2}\lambda_k^{-s/N_1}+\sum_{k=N_2+1}^{\infty}\lambda_k^{-s} \qquad (\Re(s)>1)$$
and the Dirichlet polynomials
$$B(s):=\sum_{k=1}^{N_2} (\lambda_k^{-s/N_1}-\lambda_k^{-s})$$
and 
$$C(s):=\sum_{\lambda_k\in\gammamax\setminus\{\widehat{\mathcal V}_j(L)/\widehat{\mathcal V}_1(L)\}_{j=2}^{\infty}}\lambda_k^{-s}.$$
Note that these functions are related to $\widehat{\mathcal E}_{n}(L,s)$ via 
\begin{align}\label{EpsteinABCrelation}
\widehat{\mathcal E}_{n}(L,s)=1+A(s)-B(s)-C(s).
\end{align}
We are interested in applying \cite[Theorem 2.1]{DruGarKac} to the Dirichlet series $A(s)$. Hence we need to verify that $A(s)$ satisfies the four conditions in \cite[Theorem 2.1]{DruGarKac} for almost all $L\in X_n$. The first condition, the so-called packing condition, follows from \eqref{primitivecount} which implies that for almost every $L\in X_n$, we have
\begin{equation*}
\left|\widehat N_{n,L}\big(e^{x\pm\frac{c}{x^2}}\big)-\widehat N_{n,L}\left(e^{x}\right)\right|\gg\frac{e^x}{x^2}\gg e^{(1-\ve')x}
\end{equation*}
for any $c>0$ and any $\ve'>0$ (which in turn immediately implies the packing condition for $A(s)$). The second condition, the linear independence over $\mathbb Q$ of the logarithms $\log\lambda_k$ corresponding to elements $\lambda_k\in\gammamax$, follows from the defining property of $\gammamax$. The third condition follows immediately from the fact that $\widehat{\mathcal E}_n(L,s)$ is absolutely convergent in the half-plane $\Re(s)>1$. Finally, the fourth condition (on approximation) follows for almost all $L \in X_n$ from \eqref{PRIMITIVEEPSTEIN}, the mean square estimate in Theorem \ref{MEANSQUARE} and \cite[Proposition 2.2]{DruGarKac}. Thus, for almost all lattices $L\in X_n$, we may apply \cite[Theorem 2.1]{DruGarKac} to $A(s)$. For the rest of this proof we assume that we are working with such a lattice $L\in X_n$.

Let $$g(s):=\frac {f(s)} {2 \zeta(ns)} \bigg(\frac {V_n}{2\widehat{\mathcal V}_1(L)}\bigg)^{-s}+B(s)+C(s)-1,$$
where $f(s)$ is the function given in the statement of the theorem, and let 
$$\beta:=\sup_{\Re(s)\in[\frac34,1]}\left|2\zeta(ns)\bigg(\frac{V_n}{2\widehat{\mathcal V}_1(L)}\bigg)^{s}\right|=\max_{x\in[\frac34,1]}\left|2\zeta(nx)\bigg(\frac{V_n}{2\widehat{\mathcal V}_1(L)}\bigg)^{x}\right|.$$ 
Then we can apply \cite[Theorem 2.1]{DruGarKac}, with suitable $\delta$ and $\mu$, to $A(s)$ and $g(s)$ to obtain, for a positive proportion of $t\in[0,T]$ as $T\to\infty$,
\begin{align}\label{uio} 
\max_{s \in K}  \left|A(s+it)-g(s)\right|<\frac{\ve}{4\beta}
\end{align}
and
\begin{align} \label{ineq} 
\left \| \frac{ t\log \lambda_k}{2 \pi N_1 } \right \|< \frac{\delta}{N_3}
\end{align}
for $1\leq k\leq N_2$, where $\|\cdot\|$ denotes the distance to the closest integer. Note that it follows from \eqref{abb}, \eqref{ineq} and the triangle inequality that
$$\left \| \frac{ t\log (p^n)  }{2 \pi } \right \|<\delta$$
for all $p<N_0$ and 
$$\left \| \frac{t\log (V_n/(2\widehat{\mathcal V}_1(L)))}{2 \pi} \right \| <\delta.$$
Hence, using \eqref{Truncatedzeta} and the absolute convergence of $\zeta(s)$ in $\Re(s)>1$, we obtain, for all sufficiently small choices of $\ve_1$ and $\delta$,
\begin{align}\label{INVERTEDZETAAPPROXIMATION}
\max_{s \in K} \left|\frac{(V_n/(2\widehat{\mathcal V}_1(L)))^{-s}}{2\zeta(ns)}-\frac{(V_n/(2\widehat{\mathcal V}_1(L)))^{-s-it}}{2\zeta(n(s+it))}\right|<\frac{\ve}{4\beta(\max_{s\in K}|f(s)|+1)}.
\end{align}
Similarly, we use \eqref{ineq} with a possibly smaller constant $\delta$ to conclude that
\begin{align}\label{BandCestimate}
\max_{s \in K} \left|B(s)-B(s+it)\right| <\frac{\ve}{4\beta}  \qquad \text{and}  \qquad \max_{s \in K} \left|C(s)-C(s+it)\right|<\frac{\ve}{4\beta}.
\end{align}
We stress that \eqref{INVERTEDZETAAPPROXIMATION} and \eqref{BandCestimate} holds for the same values of $t$ as \eqref{uio} and \eqref{ineq}. Throughout, we restrict our attention to such values of $t$.

In order to finish the proof it remains to put the above pieces together. Using \eqref{EpsteinABCrelation}, \eqref{uio} and \eqref{BandCestimate}, together with the triangle inequality, we obtain
\begin{multline*}
\max_{s \in K} \left|\widehat{\mathcal E}_{n}(L,s+it)-\frac {f(s)} {2 \zeta(ns)} \bigg(\frac {V_n}{2\widehat{\mathcal V}_1(L)}\bigg)^{-s}\right| \leq \max_{s \in K} \left|A(s+it)-g(s)\right|\\
+\max_{s \in K}\left|B(s)-B(s+it)\right|+\max_{s \in K}\left|C(s)-C(s+it)\right|<\frac{3\ve}{4\beta}.
\end{multline*}
Next, it follows from the above and \eqref{INVERTEDZETAAPPROXIMATION} that
\begin{align*}
&\max_{s \in K} \left|\widehat{\mathcal E}_{n}(L,s+it)-\frac {f(s)} {2 \zeta(n(s+it))} \bigg(\frac {V_n}{2\widehat{\mathcal V}_1(L)}\bigg)^{-s-it}\right|\\
&\hspace{40pt}\leq \max_{s \in K} \left|\widehat{\mathcal E}_{n}(L,s+it)-\frac {f(s)} {2 \zeta(ns)} \bigg(\frac {V_n}{2\widehat{\mathcal V}_1(L)}\bigg)^{-s}\right|\\
&\hspace{80pt}+\max_{s\in K}|f(s)|\max_{s \in K} \left|\frac{(V_n/(2\widehat{\mathcal V}_1(L)))^{-s}}{2\zeta(ns)}-\frac{(V_n/(2\widehat{\mathcal V}_1(L)))^{-s-it}}{2\zeta(n(s+it))}\right|<\frac{\ve}{\beta}.
\end{align*}
Finally, we use \eqref{PRIMITIVEEPSTEIN} to arrive at
\begin{multline*}
\max_{s \in K} \left|E_n\left(L,\frac {n(s+it)}2\right)-f(s)\right|\\
\leq\beta\max_{s \in K} \left|\widehat{\mathcal E}_{n}(L,s+it)-\frac {f(s)} {2 \zeta(n(s+it))} \bigg(\frac {V_n}{2\widehat{\mathcal V}_1(L)}\bigg)^{-s-it}\right|<\ve,
\end{multline*}
which is the desired conclusion.
\end{proof}

\end{document}